\documentclass[10pt]{amsart}
\usepackage{amsmath,amsfonts,amsthm,amssymb,amscd, soul}
\def\classification#1{\def\@class{#1}}
\classification{\null}

\usepackage[a4paper, total={6in, 8in}]{geometry}
\DeclareFontFamily{OT1}{rsfs}{}
\DeclareFontShape{OT1}{rsfs}{n}{it}{<-> rsfs10}{}
\DeclareMathAlphabet{\mathscr}{OT1}{rsfs}{n}{it}

\newcommand{\R}{{\mathbb R}}

\newcommand{\C}{\mathbb{C}}

\newtheorem{theorem}{Theorem}
\newtheorem{conjecture}{Conjecture}
\newtheorem{proposition}[theorem]{Proposition}
\newtheorem{lemma}[theorem]{Lemma}
\newtheorem{corollary}[theorem]{Corollary}
\theoremstyle{remark}
\newtheorem{remark}[theorem]{Remark}

\title{
On distinct cross-ratios and related growth problems}

\author{Misha Rudnev}
\address{Misha Rudnev, Department of Mathematics, University of Bristol,
  Bristol BS8 1TW, United Kingdom}
\email{m.rudnev@bristol.ac.uk}

\begin{document}
\begin{abstract} It is shown that for a finite set $A$ of four or more complex numbers, the cardinality of the set $C[A]$ of all cross-ratios generated by quadruples of pair-wise distinct elements of $A$ is $|C[A]|\gg |A|^{2+\frac{2}{11}}\log^{-\frac{6}{11}} |A|$ and without the logarithmic factor in the real case. The set $C=C[A]$ always grows under both addition and multiplication.

The cross-ratio arises, in particular, in the study of the open question of the minimum number of triangle areas, with two vertices in a given non-collinear finite point set in the plane and the third one at the fixed origin. The above distinct cross-ratio bound implies a new lower bound for the latter question, and enables one to show growth of the set $\sin(A-A),\;A\subset \R/\pi\mathbb Z$ under multiplication. It seems reasonable to conjecture that more-fold product, as well as sum sets of this set or $C$ continue growing ad infinitum.\end{abstract}
\maketitle

\section{Introduction and results}  Given an ordered quadruple of pair-wise distinct elements $a,b,c,d\in F P^1$, the projective line over a field $F$, their cross-ratio is defined as
\begin{equation}\label{crr}
[a,b,c,d] := \frac{(a-b)(c-d)}{(a-c)(b-d)}.
\end{equation}
Further $F$ is the real or complex field, and we write $C[A]$ for the set of cross-ratios generated by all quadruples of pair-wise distinct elements $a,b,c,d$ in some finite set of scalars $A\subset F$ of cardinality $|A|\geq 4$. The standard notations $X\ll Y$, equivalently $X=O(Y)$ or $Y\gg X$, are used to suppress absolute constants in inequalities. Besides, $Y\gtrsim X$ means $Y\gg X\log^{-c}|A|$ for some $c>0$; logarithms are base $2$.

\medskip
The cross-ratio represents the key projective invariant, in particular being  invariant under the action of  the group $PSL_2(F)$ on $FP^1$, acting as M\"obius transformations, for \eqref{crr} is invariant under a simultaneous shift, dilation, or inversion of all its variables. Conversely, two cross-ratios $[a,b,c,d]$ and $[a',b',c',d']$  are equal only if there is a M\"obius transformation, mapping $a$ to $a'$, $b$ to $b'$, $c$ to $c'$, and $d$ to $d'$.  That is, there exist scalars $\alpha,\beta,\gamma,\delta$, with $\alpha\delta - \beta\gamma=1$, $a'= \frac{\alpha a+\beta}{\gamma a +\delta},$ and the same for $b,c$ and $d$.

Furthermore, given a family of collinear points in the projective plane, their cross-ratios (when the points are viewed as scalars and the line supporting them as $FP^1$) are invariant with respect to projective transformations of the plane. 

The cross-ratio appears to be closely related to the the so-called sum-product phenomenon, with the landmark Erd\H os-Szemer\'edi conjecture \cite{ES}, which in the context of  real numbers suggests that
$$
|A+A|+|A\cdot A| \gtrsim |A|^{2},
$$
where as usual, $A+A$ and $A\cdot A=AA$ denote the set of all pair-wise sums and, respectively, products of elements of $A$. The notations $kA,\,A^k$, for  integer $k\geq 2$, $A^{-1}=1/A,\;A/A$ are used, respectively, for the $k$-fold sum or product set of $A$,  the set of finite reciprocals and ratios of elements of $A$; for a complex-valued function $f$, $f(A)$ denotes the range of the restriction of $f$ to $A$.

The observation that the sum-product conjecture can be viewed and explored as a projective plane problem may have been first made and followed upon in a paper by Chang and Solymosi \cite{CS}.

The cross-ratio has appeared in the study of two more open questions. One is the following conjecture.
\begin{conjecture}\label{c1} For a non-collinear point set $E\in \R^2$, the cardinality of the set $\omega[E]$ of the values of the symplectic form $\omega$ on all pairs of vectors -- points of $E$ always satisfies $|\omega[E]|\gtrsim |E|$.\end{conjecture}

In other words, $\omega[E]$ is the set of all oriented areas of triangles $Oab$ with $a,b\in E$ and a fixed origin $O$.  It is not clear, to whom Conjecture \ref{c1} can be formally attributed, nor whether the $\gtrsim$ symbol in its above formulation may be replaced by $\gg$. In the past Conjecture \ref{c1} may have been regarded as somewhat second to the well-known conjecture of Erd\H os on the minimum number of distinct distances determined by a plane set of $N$ points. The latter question was resolved by Guth and Katz \cite{GK}, but the one on the minimum number of distinct values of any non-degenerate bilinear form on a set of $N$ points is wide open. See \cite{IRR}, \cite{57} for the best known results.

The other question was raised in a recent paper by Balog, Roche-Newton and Zhelezov \cite{BRZ}. Let the set of scalars $A$ have at least three elements and $D=A-A$. \begin{conjecture}\label{c2} For any positive integer $n$, there is a universally and polynomially related $m$, such that
$|D^m|\geq |D|^n$.\end{conjecture} 
Let us further refer to growth as above as {\em polynomial, ad infinitum.}

\medskip
Shkredov \cite{Sh} was able to show a {\em local} instance of this conjectured {\em global} growth by proving that $|DD|\gtrsim |D|^{1+\frac{1}{12}}$ (see \cite{57} for the positive characteristic analogue of this result). Proposition \ref{cor2} herein establishes a slightly weaker bound for the set $\sin D \cdot \sin D$ (in the case $A\subset \R/\pi \mathbb Z$). Our interest in the latter set stems from the fact that it seems plausible that the set $\sin D$ should necessarily grow polynomially ad infinitum both under addition and multiplication (which is certainly not the case with $D=A-A$). Apart from local growth of the latter set under multiplication, in the end of this note we show logarithmic global growth, conditional on  Conjecture \ref{c1}.

\medskip
In addition to Conjecture \ref{c1}, it appears reasonable to suggest that one always has $|C[A]| \gtrsim |A|^3$. There are approximately $|A|^3$ distinct cross-ratios if $A$ is a geometric progression, and there is no evidence that the claim cannot be strengthened to $|C[A]| \gg |A|^3$. What has been known so far is much weaker: the ``basic'' estimate that $|C[A]| \gg |A|^2$ is ten years old, due to Solymosi and Tardos \cite{ST}. More precisely, the latter estimate is implicit in \cite[Theorem 4]{ST}; some five years thereafter it became independently rediscovered by Jones \cite{J}.

The main result here is the following increment. 

\begin{theorem}  \label{mish}  Let $A$ be a finite set of at least four  complex numbers. Then $|C[A]|\gtrsim |A|^{2+\frac{2}{11}}.$ For real $A$, one has $|C[A]|\gg |A|^{2+\frac{2}{11}}.$    \end{theorem}

Theorem \ref{mish} breaks a ``psychological'' threshold of proving super-quadratic growth (that is beyond $|A|^2$) for the range of the rational function $[a,b,c,d]$ of four variables in $A$. The aforementioned paper \cite{BRZ} is one of the few in the literature that establishes superquadratic growth, dealing with a six-variable polynomial, corresponding to the set $(A-A)(A-A)(A-A)$. Hopefully, what \cite{BRZ} call {\em superquadratic expanders} are steps towards proving {\em global} {\em ad infinitum} sum-product type growth, preferably polynomial, over the reals (as well as other fields, with a properly adjusted meaning of {\em ad infinitum}). Its instance was conjectured in \cite{BRZ} and probably has inherent links  to expander graph theory. Note that over the integers, and hence rationals, there is an outstanding result on polynomial ad infinitum growth by Bourgain and Chang \cite{BC}.

\medskip
Theorem \ref{mish} can be used to derive a few corollaries. Symmetry properties of $C$ discussed below underlie the following claim.
\begin{proposition}  \label{cor1}  The set $C=C[A]$ satisfies the following growth estimates:
$$
|C\pm C|,\,|C\cdot C|\;\gtrsim \; |C|^{\frac{24}{19}}.
$$
\end{proposition}

\begin{corollary}  \label{cor2}  For  $A\subset \R/\pi\mathbb Z$, let $D=A-A$ and  $\sin D = \{\sin d:\,d\in D\}.$ Then
$$
|\sin D\cdot\sin D|\gg |D|^{1+\frac{7}{110}}.
$$
\end{corollary}

Besides, Theorem \ref{mish} yields a new bound apropos of Conjecture \ref{c1}.
\begin{corollary}  \label{misha}  Let $E$ be a finite set of points in the complex plane, not supported on a single line through the origin. Then $|\omega[E]|\gtrsim |E|^{\frac{96}{137}}$, and $|\omega[E]|\gg |E|^{\frac{96}{137}}$ in the real case.\footnote{The previous bound $|C(|A|)|\gg|A|^2$ yields the exponent $\frac{9}{13}$ for $|\omega[E]|$, see \cite{IRR}.}\end{corollary}

In the last section of this note we state and prove a few more results relating Conjecture \ref{c1} and growth under addition/multiplication, some of which are conditional on the conjecture.

\section{Proofs of main results}
It suffices to present the proofs in the complex setting: they rely primarily on the forthcoming Theorem \ref{SoTa} due to Solymosi and Tardos \cite{ST}, with in particular, the values of  hidden constants being quite reasonable\footnote{In the real setting Theorem \ref{SoTa} goes back to the original Szemer\'edi-Trotter theorem \cite{SzT} and its adaptation for curves by Pach and Sharir \cite{PS}.}. 
In the special case of reals there is a stronger result by Sharir and Solomon \cite{SS}, it is using this result that enables one to get rid of the logarithmic factor.

The rest of the arguments do not distinguish between real and complex numbers, unless stated. Thus ``a finite set'' means ``a finite set of real or complex numbers'', ``the plane'' means $\R^2$ or $\C^2$.

\subsubsection*{Preliminaries} We start out with the incidence and additive combinatorics tools required for our purposes and then outline the proof strategy.

\begin{theorem} \label{SoTa} Let $A$ be a finite set.

(i)  A M\"obius transformation $\tau$ is called $k$-rich if there are at least $k$ distinct quadruples of pair-wise distinct points $a,b,c,d\in A$ that $\tau$ maps into $A$. Let $M_k(A)$ be the set of $k$-rich M\"obius transformations on $A$. Then
\begin{equation}\label{mk}
|M_k(A)| \ll \frac{|A|^6}{k^5},
\end{equation}
for $A\subset \mathbb C$ and
\begin{equation}\label{mks}
|M_k(A)| \ll \frac{|A|^6\log k}{k^{\frac{11}{2}}},
\end{equation}
for $A\subset \mathbb R.$
(ii) The number of incidences $I$ between the point set $A\times A$ and a set of $N$ lines in the plane satisfies
$$I\ll |A|^{\frac{4}{3}}N^{\frac{2}{3}}+|A|^2 + N.$$\end{theorem}
Strictly speaking, the claim (ii) is not explicit in \cite{ST} but follows immediately from Theorem 3 therein, while (i) is  \cite[Theorem 4]{ST}. 


\begin{remark} \label{remarka} The argument in \cite{ST} also holds if the plane set of lines in claim (ii) of Theorem \ref{SoTa} is replaced by hyperbolae (corresponding to  M\"obius transformations) in the form $xy-\alpha x+\delta y -\beta=0$, where the finite set of parameter values $(\alpha,\delta,\beta)$ comes from a two, rather than three-dimensional family, that is when  two distinct points may belong to no more than two hyperbolae in the family. Then one has the following estimate 
\begin{equation}\label{mkk}
|M_k(A)| \ll  \frac{|A|^4}{k^3},
\end{equation}
for the number $M_k$ of $k$-rich lines or hyperbolae.
This is a particular case of the Szemer\'edi-Trotter theorem for real curves, which follows from the renown crossing-number proof of Sz\'ekely \cite{Sz}, but the -- easy and elegant --  arguments in \cite{ST} work in the complex setting, being based on M\"obius transformations in $\C\cong \R^2$ mapping circles into circles and either preserving or swapping their interior and exterior.\end{remark}

\medskip
We will also use the well-known Pl\"{u}nnecke inequality, mostly in the standard Ruzsa form  \cite{Ruz}, and once  ``thin'' form by Katz and Shen \cite{KS}. The state-of-the-art proof is due to Petridis \cite{Pet}. 

\begin{lemma} \label{Plun} Let $A$ be a subset of an abelian group $(G,+)$, suppose $|A+A|=K|A|$  and $k\geq 1,\,l\geq 0$ are integers. Then 

(i) one has ${\displaystyle |kA-lA| \leq  K^{k+l}|A|}$,

(ii) there exists $A'\subseteq A$, with $|A'|\geq .9|A|$ and such that $|A'+kA| \ll K^k|A|.$
\end{lemma}

\medskip
A few preliminaries about M\"obius transformations. Let us for the future call two $s$-tuples of pair-wise distinct elements $a,b,c,d,\ldots$ and $a',b',c',d',\ldots$ of $A$ {\em congruent} if there is a M\"obius transformation $\tau$ taking $a$ to $a'$, $b$ to $b'$, $c$ to $c'$, $d$ to $d'$, and so on. Clearly, for $s=3$ any two triples of distinct elements of $A$ are congruent. 

Note that the statement (i) of Theorem \ref{SoTa} implies that $|C[A]|\gg |A|^2$. Indeed, it follows from \eqref{mk} (by analogue with forthcoming estimate \eqref{pent}) that the the number $Q$ of congruent quadruples is $Q=O(|A|^6)$. Two congruent quadruples yield the same cross-ratio $[a,b,c,d]=[a',b',c',d']$. Therefore, by the Cauchy-Schwarz inequality, the number of congruence classes 
$$
|C[A]| \gg \frac{|A|^8}{Q}\gg |A|^2.
$$
The problem with the estimate for the quantity $Q$ coming from Theorem \ref{SoTa} is that this estimate is hardly sharp (for quadruples of pair-wise {\em distinct} elements of $A$). The statement (i) of Theorem \ref{SoTa} is ideally suited for estimating the number of pairs of congruent {\em pentuples} in $A$, which we denote as $P$. After dyadic summation in the M\"obius transformation richness parameter $k$ we use \eqref{mk}, getting
\begin{equation}
P\ll \sum_{j=2}^{\lceil \log |A|\rceil } |M_{2^j}(A)| (2^j)^5 \ll |A|^6\log |A|.
\label{pent}
\end{equation}
In the real case we use the stronger bound \eqref{mks}, which  improves the latter estimate to
\begin{equation}
P\ll |A|^6.
\label{pentt}
\end{equation}

This estimate is sharp, say for $A=[1,\ldots,|A|]$ for it counts, in particular, pentuples of points in $A\times A$, which are collinear on non-horizontal and non-vertical lines, corresponding to affine transformations (which are also M\"obius transformations).

Unfortunately, the estimate for the quantity $P$ itself is not directly suitable for getting a lower bound on $|C[A]|$. To benefit by it, one needs to take advantage of certain algebraic relations that the cross-ratio satisfies. Properties of this type have earlier been used in the paper by Iosevich, Roche-Newton \cite{IRR} and the author, but the particular way of using the ones in this paper, namely the forthcoming formulae \eqref{one}, \eqref{two}, is due to Shkredov \cite{Sh}  who dealt with the set 
\begin{equation} \label{ar} R[A]:=\left\{\frac{a-b}{a-c}=[a,b,c,\infty]:\,\mbox{distinct }a,b,c\in A \right\}\end{equation} of cross-ratios pinned at infinity. 

These algebraic properties of the cross-ratio, together with trying to get the best out of incidence theorems constituted the leitmotif of the recent hefty paper by Murphy, Petridis, Roche-Newton, Shkredov, and the author, which dealt with similar questions in fields of positive characteristic. 
The following proof of Theorem \ref{mish} morally follows the scheme in \cite[proof of Theorem 3]{57}, being technically easier, for one does not have to worry about the characteristic of the field.

We start with a well-known fact, with a quick proof for completeness sake.
\begin{lemma} Suppose $[a,b,c,d]=[a',b',c',d']$ for two quadruples of pair-wise distinct scalars. Then there is a M\"obius transformation $\tau$ taking $a$ to $a'$, $b$ to $b'$, $c$ to $c'$ and $d$ to $d'$.
\label{warmup}\end{lemma}

\begin{proof} The condition $[a,b,c,d]=[a',b',c',d']$, using the fact that elements within each quadruple are pairwise distinct, can be rearranged as 
$$\det\left(\begin{array}{cccc} 1&1&1&1\\a&b&c&d\\a'&b'&c'&d'\\aa'&bb'&cc'&dd'\end{array}\right)=0.$$ This means that the rows, from top to bottom, of the above matrix are linearly dependent with some coefficients, say $-\beta,\,-\alpha,\,\delta,\gamma$, respectively. It follows that the sought for M\"obius transformation $\tau$ is given by the matrix $\left( \begin{array}{cc}\alpha&\beta\\ \gamma &\delta\end{array}\right)$, and the four points $(a,a'),\,(b,b'),\,(c,c'),\,(d,d')$ in the plane all lie on the hyperbola
$y(\gamma x+ \delta)-(\alpha x+\beta) = 0.$
If the determinant of the latter $2\times 2$ matrix is zero, this means that the hyperbola degenerates into a horizontal or vertical line, which is not allowed, since elements in each quadruple are pair-wise distinct. Thus one can normalise the matrix by unit determinant. \end{proof}

\begin{proof}[Proof of Theorem \ref{mish}]

Consider a pentuple $a,b,c,d,e\in A$, where all elements are pair-wise distinct (for the rest pair-wise distinctness is implicit). Set 
$$
x=[a,b,c,d],\qquad y = [a,b,c,e].
$$
Knowing $x,y$ fixes the pentuple up to a congruence via some M\"obius transformation $\tau$. Indeed, if two pentuples are congruent, they yield the same values of cross-ratios $x,y$. Conversely if $x=[a,b,c,d]=[a',b',c',d']$ and $y=[a,b,c,e] = [a',b',c',e']$, then, by Lemma \ref{warmup} some M\"obius transformation  $\tau_x$ takes $a,b,c,d$ to, respectively, $a',b',c',d'$  and some $\tau_y$ takes $a,b,c,e$ to, respectively, $a',b',c',e'$. Since $\tau_x$ and $\tau_y$ agree on $a,b,c$, they are equal.

Furthermore, one has the following algebraic relations:
\begin{equation}\label{one}
\frac{x}{y} = \frac{(a-b)(c-d)}{(a-c)(b-d)}  \frac{(a-c)(b-e
)}{(a-b)(c-e)} = \frac{(d-c)(b-e)}{(d-b)(c-e)} =  [d,c,b,e]\in C[A],
\end{equation}
as well as
\begin{equation}\label{two}
\frac{x-1}{y-1} =  \frac{(a-d)(c-b)}{(a-c)(b-d)}  \frac{(a-c)(b-e
)}{(a-e)(c-b)} = \frac{(a-d)(e-b)}{(a-e)(d-b)} = [a,d,e,b]\in C[A].
\end{equation}

Hence, define $G\subset C[A]\times C[A]$ as follows:

$$G = \{(x,y): \;\exists \,a,b,c,d,e\in A:\, x=[a,b,c,d], y=[a,b,c,e]\}$$ and set
$p(x,y)$ to be the number of congruent pentuples $a,b,c,d,e$ yielding some $(x,y)\in G$. Then one has the identity, cf. \eqref{pent}:
$$
\sum_{(x,y)\in G} p^2(x,y) = P \ll |A|^6 \log|A|.
$$ 
The same estimate without the logarithm holds if we use \eqref{pentt} in the real case.

By the Cauchy-Schwarz inequality, it follows that
\begin{equation}
|A|^{10}\ll
\left(\sum_{(x,y)\in G} p(x,y)\right)^2\leq   |G| \sum_{(x,y)\in G} p^2(x,y)  \ll  |G||A|^6\log|A|,
\label{cs}\end{equation}
and without the logarithm in the real case.

It remains to get an estimate on $|G|$. Note that a trivial lower bound for $|G|$ is $|C[A]||A|$, since for every $x \in C[A]$ we can fix a representation of $x$, so in particular, $a,b,c$, and vary $e$ over $A$, thus getting different values of $y$. If this were also an upper bound on $|G|$, it  would yield the Holy Grail inequality $C[A] \gg |A|^3/\log|A|,$ and, in fact, without the log in the real case.

One has to be content with less. Partition the set $G$ into $G'$ and $G''$, where $G'$ is defined as the subset of $(x,y)\in G$, such that for every $x\in C[A]$ there are at most $tC[A]|$ distinct values of $y\in C[A]$ with $(x,y)\in G$. The small $0<t<1$ is to be chosen later.

Clearly
\begin{equation}\label{small}
|G'|\leq t |C[A]|^2.
\end{equation}
To estimate the size of the complement $G''$ of $G'$ in $G$, let us use \eqref{one} and \eqref{two}.
Namely, by \eqref{two}, we have, for some $r\in C[A]$ that
$$
x-1 = r(y-1).
$$
Furthermore, by \eqref{one} and the definition of $G''$, for every $x$ there are at least $t |C[A]|$ values of pairs $(r',y')\in C[A]\times C[A]$, such that $x=r'y'$. Thus
for every $(x,y)\in G''$ there are at least $t |C[A]|$ solutions of the four-variable equation 
$r'y' - r(y-1) =1$, with the variables in  $C[A]$. I.e.,
 $$
|G''|\leq \frac{1}{t |C[A]|} I,
$$
where 
$$I=| \{(r,r',y,y')\in C[A]\times C[A]\times C[A]\times C[A]:\; r'y' - r(y-1) =1\}|. $$
By the claim (ii) of Theorem \ref{SoTa} we have $I\ll |C[A]|^{\frac{8}{3}},$ so $|G''|\leq t^{-1}|C[A]|^{\frac{5}{3}}.$

Optimising with \eqref{small} leads to choosing $t = |C[A]|^{-\frac{1}{6}}$, thus 
$$
|G|\ll |C[A]|^{\frac{11}{6}}.
$$
Substituting this into \eqref{cs} completes the proof of Theorem \ref{mish}.
\end{proof}

\begin{proof}[Proof of Proposition \ref{cor1}]
The claim follows from the following properties: $C-1 = -C$, as well as $1/C=C$. Thus $|CC|=|C(C-1)|$, as well as $|C\pm C| = |C\pm 1/C|$. It is well known that for any real or complex finite set $C$,
$$
|C(C\pm1)|,\,|C\pm1/C|\;\gtrsim \;|C|^{\frac{24}{19}},
$$
and this estimate relies only on the claim (ii) of Theorem \ref{SoTa} (see also Remark \ref{remarka}) and additive combinatorics arguments. See e.g. \cite[Theorem 1.2]{JRN}, \cite[Corollary 16]{Sh1}.
\end{proof}

\begin{proof}[Proof of Corollary \ref{cor2}]

Take $l$ to be a line not through the origin $O$, let $A$ be a set of points on $l$ which we identify with a set of scalars, viewing $l$ as a copy of $\R.$
Given four points $a,b,c,d\in A,$ observe that their cross-ratio
\begin{equation}
[a,b,c,d] = \frac{\sin\widehat{Oab}\,\sin\widehat{Ocd}}{\sin\widehat{Oac}\,\sin\widehat{Obd}},
\label{tri}\end{equation}
where, say $Oab$ is the signed angle from the origin between the directions to $a$ and $b$ as points in $\R^2$ on $l$. Let $\Phi\subset \R^2/ \pi\mathbb Z$ be the set of directions from $O$ to the points of $A$, then $|\Phi-\Phi|=|A-A|$ and \begin{equation}\label{incl}C[A]\subseteq (\sin(\Phi-\Phi)\cdot \sin(\Phi-\Phi))/(\sin(\Phi-\Phi)\cdot \sin(\Phi-\Phi)).\end{equation}

Let $S=\sin(\Phi-\Phi)\setminus\{0\}$. Suppose $|SS|=K|S|$. By claim (i) of Lemma \ref{Plun} and inclusion \eqref{incl},
\begin{equation}\label{alm} |SC|\leq K^5|S|.\end{equation}

Consider the set of lines in $\R^2$ in the form $y=(s^{-1}x -1)s'$, with $s,s'\in S$ and its incidences with the set $SC\times (-SC)$. Since $C-1 = -C$, every line supports at least $|C|$ points, with $x=sc$, for any $c\in C$.

Hence, by the claim (ii) of Theorem \ref{SoTa}
\begin{equation}\label{stap}
|S|^2|C|\ll |SC|^{\frac{4}{3}}|S|^{\frac{4}{3}} + |SC|^{2}.
\end{equation}
It follows that 
\begin{equation}\label{itr}
|SC|\gg \min \left( |C|^{\frac{3}{4}}|S|^{\frac{1}{2}}, \; |S||C|^{\frac{1}{2}}\right).
\end{equation}
Thus by \eqref{alm}
$$
K^5 \gg \min\left( |C|^{\frac{3}{4}}|S|^{-\frac{1}{2}}, |C|^{\frac{1}{2}}\right).
$$
Using now Theorem \ref{mish} (in the real case) and the worst possible case $|S|\gg|A|^2$, we conclude that  
$K \gg |A|^{\frac{7}{55}} \geq |D|^{\frac{7}{110}},$
as claimed.
\end{proof}

\begin{remark} 
Observe that in the above proof of Corollary \ref{cor2} the only thing that has been used about the set $C$ is its mirror symmetry with respect to $\frac{1}{2}$. An additive version of the estimate \eqref{stap} is easily obtained by using symmetry of $C$ to taking the reciprocal and Remark \ref{remarka}. Thus one can show, for instance, that
$$ \begin{aligned}
| \,[ C+(C-C) ] \cdot [ C+(C-C) ] \,| & \gtrsim |C+(C-C)|^{\frac{24}{19}} \gg \left(|C|^{3/4}|C-C|^{1/2}\right)^{\frac{24}{19}} \\
& \gtrsim |C|^{\left(\frac{3}{4}+\frac{12}{19})\cdot\frac{24}{19}\right)} 
\\ & \gtrsim |A|^{\frac{15120}{3971}},\end{aligned}
$$
the latter lower bound thus being valid for the set $|5CC-4CC|$ due to the obvious inclusion of the left-hand side.

Similar bounds illustrating superquadratic growth can be derived using the set $R[A]$ of pinned cross-ratios determined by \eqref{ar} (via a M\"obius transformation they can be pinned to any point, not necessarily infinity). The latter set satisfies the weaker then in Theorem \ref{mish} yet  sharp up to the logarithmic factor bound $|R[A]|\gg |A|^2\log^{-1}|A|,$ which would change the rightmost term in the latter inequality to  $|A|^{\frac{1260}{361}}$.

This contributes to the leitmotif of this note that sum-product sets of cross-ratio sets $C[A],R[A]$ are more prone to grow than of just $A$, where one may have the  upper bound $k|AA|\leq C(k)|A|^2$. \end{remark}

For the proof of  Corollary \ref{misha} see \cite[proof of Theorem 1, Lemma 10]{IRR}.

\section{More on Conjecture \ref{c1}} Suppose, one believes in the polynomial ad infinitum growth by multiplication of any nonempty set $C[A]$ for $A\subset \R$, as well as  $\sin(\Phi-\Phi)$, where $\Phi\subset \R/\pi \mathbb Z$ is the set of more than two angular directions from the origin, as well as in  Conjecture \ref{c1}. Take $\Phi=\Phi[E]$, the set of directions from the origin to the points of some point set $E\subset \R^2$. Then the set $\omega[E]$  of triangle areas determined by pairs of elements of $E$ also should grow polynomially ad infinitum under multiplication, provided that $|E| $ and $|\Phi|$ are polynomially related. This follows by the multiplicative Pl\"unnecke inequality.  Indeed, formula \eqref{tri} can be interpreted as the cross-ratio of four triangle areas, rooted at the origin, the triangles being defined by any four points $e_a,e_b,e_c,e_d\in E$, lying respectively in the directions identified by $a,b,c,d$. 

If $|\Phi[E]|=O(1),$ however, then one can simply take a long geometric progression $T$ and take $E$ as a union of $|T||\Phi|$ points on $|\Phi|$ lines, the distance from each point to the origin being in $T$, in which case clearly $|\omega[E]\cdot\omega[E]|\ll |\omega[E]|$.  This brief section contains some easy partial or conditional results in this vein.

The first one follows from Theorem \ref{mish} and Proposition \ref{cor1} and shows that if the set $\omega[E]$ has a $O(1)$ multiplicative growth constant, then $E$ should determine fairly few, that is $|\Phi[E]|\lesssim |\omega[E]|^{\frac{11}{36}}$ directions.

\begin{corollary}\label{cor3} Suppose  $|\omega[E]\cdot\omega[E]|=K |\omega[E]|$, for some $K \geq 1$.
Then for any $\epsilon >0$, $|\Phi[E]|\ll K^{O(-\log \epsilon)} |\omega[E]|^{\frac{11}{36}-\epsilon}.$\end{corollary}

\begin{proof} The proof is basically the same as of Corollary \ref{cor2}. Assume that $E$ determines at least four directions or there is nothing to prove. Set $\omega=\omega[E]$, let $l$ be any line not through the origin and $A$ the set of scalars on $l$, corresponding to intersections of $l$ with lines through the origin in the directions from $\Phi$, so $|A|=|\Phi|$. Set  $C=C[A]$. It follows by \eqref{tri} that
$C\subseteq \omega^2/\omega^2$,  and for integer $k\geq 2$, $C^k\subseteq \omega^{2k}/\omega^{2k}$. Hence, we can use claim (i) of Lemma \ref{Plun}, and iterate the estimate \eqref{itr} above to yield the lower bound for $|C^k|$, with $S=C^{k-1}$ which approaches $|C|^{\frac{3}{2}}$ at exponential rate.

The claim follow by invoking Theorem \ref{mish} to bound $|C|^{\frac{3}{2}}\gtrsim |\Phi|^{\frac{36}{11}}.$
\end{proof}

The next, and final statement is conditional on Conjecture \ref{c1}. It shows that for any sufficiently large set of directions $\Phi\subset \R/\pi\mathbb Z$, the set $\sin(\Phi-\Phi)$ cannot be a subset of a polynomially long approximate geometric progression and provides evidence of  logarithmic (rather than polynomial) ad infinitum growth of the latter set under multiplication. 
\begin{proposition} Assume Conjecture \ref{c1} and suppose that $\Phi=\Phi[E]$ is the set of directions of the point set $E\subset \R^2$. For any $M>1$, $0<\epsilon<\frac{1}{2M}$ and  $|\Phi|>C(\epsilon)$, for some constant $C(\epsilon)$, there is no set $T$, such that $|T|=|\Phi|^M$, $|TT|\leq |T|^{1+\epsilon},$ and $T\supseteq \sin(\Phi-\Phi)$.

For integer $n\geq 0$,  one has $|(\sin(\Phi-\Phi)|)^{2^n}| \geq |\Phi|^{n}, $ where $|\Phi|>C(n)$, for some constant $C(n)$.

\end{proposition}
\begin{proof} Suppose, for contradiction, that a putative $T$ exists. Refine $T$ to its large subset, such that  after refinement we have, the ``thin'' Pl\"unnecke bound 
$|TT\cdot \sin(\Phi-\Phi)|\ll |T|^{1+2\epsilon}$, cf. claim (ii) of Lemma \ref{Plun}. Take a point set $E$, comprising $|\Phi||T|$ points, lying in $|\Phi|$ directions from the origin, whose distances from the origin are in the set $T$. Since $\omega(E)=TT\cdot \sin(\Phi-\Phi)$,  for  sufficiently large $|\Phi|$ there is a contradiction with Conjecture \ref{c1}.

To prove the second statement of the proposition, amply true for $n=0$, let $n\geq 1$ and set $S=\sin(\Phi-\Phi)$, $T=(\sin(\Phi-\Phi))^n$. Construct the point set $E$ of $|\Phi||T|$ points in the plane as above, so $\Phi$ is the set of directions and $T$ the set of distances to the origin. Then $\omega[E]=S^{2n+1}.$
Assuming Conjecture \ref{c1} one must have $|S^{2n+1}| \gtrsim |S^{n}||\Phi|,$ that is sightly more than doubling $n$ increases the size of the set $(\sin(\Phi-\Phi))^n$ by at least a factor $|\Phi|$. The statement follows by induction in $n$, for sufficiently large $|\Phi|$.

\end{proof}

\section*{Acknowledgement} The author thanks O. Roche-Newton and G. Petridis for more than a ``welcome feedback'' on an earlier version of this note and Noam Solomon for clarifications concerning the use of the advanced bound  \eqref{mks}.

\end{document}